\documentclass[12pt, a4paper, oneside]{article}
\usepackage[top=1in, bottom=1.1in, left=0.9in, right=0.8in]{geometry}
\usepackage[parfill]{parskip}
\usepackage{amsmath}
\usepackage{amsthm}
\usepackage{graphicx}
\usepackage{amssymb}
\usepackage{epstopdf}
\usepackage{setspace}
\usepackage{authblk}
\usepackage{enumerate}
\usepackage{lmodern}
\usepackage[hypcap]{caption}
\usepackage{color}
\definecolor{dblue}{rgb}{0,0,0.45}
\definecolor{red}{rgb}{0.7,0,0}

\usepackage[english]{babel}
\usepackage{fancyhdr}
\addto\captionsenglish{}
\addto\captionsenglish{}
\addto\captionsenglish{}
\addto\captionsenglish{}
\addto\captionsenglish{}

\newlength\longest

\newcommand{\cM}{{\mathcal M}}

\newcommand{\N}{{\mathbb N}}
\newcommand{\R}{{\mathbb R}}

\begin{document}

\newtheorem{thm}{Theorem}[section]
\newtheorem{lem}[thm]{Lemma}
\newtheorem{prop}[thm]{Proposition}
\newtheorem{cry}[thm]{Corollary}
\theoremstyle{definition}
\newtheorem{dfn}[thm]{Definition}
\newtheorem{example}[thm]{Example}
\theoremstyle{remark}
\newtheorem{rem}[thm]{Remark}
\numberwithin{equation}{section}

\title{Proper Inclusions of Morrey Spaces}
\author{Hendra Gunawan, Denny I. Hakim, and Mochammad Idris}
\affil{Department of Mathematics, Bandung Institute of Technology\\
Bandung 40132, Indonesia}
\date{}

\maketitle

\begin{abstract}
In this paper, we prove that the inclusions between Morrey spaces, between weak
Morrey spaces, and between a Morrey space and a weak Morrey space are all proper.
The proper inclusion between a Morrey space and a weak Morrey space is established
via the unboundedness of the Hardy-Littlewood maximal operator on Morrey spaces
of exponent 1. In addition, we also give a necessary condition for each inclusion.
Our results refine previous inclusion properties studied in \cite{GHLM}.


\end{abstract}

\section{Introduction}

Morrey spaces were first introduced by C.B. Morrey in \cite{Morrey} in relation
to the study of the solution of certain elliptic partial differential equations.
For $1 \le p \le q < \infty$, the {\it Morrey space} $\cM^p_q=\cM^p_q(\R^d)$ is
defined to be the set of all $f\in L^p_{\rm loc}(\R^d)$ such that
\[
\|f\|_{\cM^p_q}:=\sup_{a \in \R^d,\ r > 0}\ |B(a,r)|^{\frac{1}{q}}
\left( \frac{1}{|B(a,r)|}\int_{B(a,r)} |f(y)|^p\ dy\right)^{\frac{1}{p}}<\infty.
\]
Here, $B(a,r)$ is an open ball centered at $a$ with radius $r$, and $|B(a,r)|$
denotes its Lebesgue measure. Notice that, when $p=q$, one can recover the Lebesgue
space $L^p=L^p(\R^d)$ as the special case of $\cM^p_q$. See \cite{Peet} for
various spaces related to Morrey spaces. Many researchers have proved the
boundedness of classical integral operators on Morrey spaces and their
generalizations. See, for instance, \cite{Chiarenza, Gunawan} and the references
therein.

Concerning the Hardy-Littlewood maximal operator (defined in Section 3), one may prove
its boundedness on Morrey spaces using the inclusion $\cM_q^p \subseteq \cM^1_q$.
In general, we have the following inclusions of Morrey spaces
\begin{align*}
L^q = \cM^{q}_q \subseteq \cM^{p_2}_q \subseteq \cM^{p_1}_q \subseteq \cM^1_q
\end{align*}
provided that $1\le p_1\le p_2\le q<\infty$.
These inclusions may be obtained by applying H\"older's inequality. Note that, for
$1\le p_2<q<\infty$, we have $f(x):=|x|^{-\frac{d}{q}} \in\cM^{p_2}_q \setminus \cM^{q}_q$. This
tells us that the inclusion $\cM^{q}_q \subseteq \cM^{p_2}_q$ is proper
for $1\le p_2<q<\infty$.

Besides the `strong' Morrey spaces, we also have weak Morrey spaces whose definitions
are given as follows:

\bigskip

\begin{dfn}
Let $1\le p\le q<\infty$. A measurable functions $f$ on $\R^d$ is said to belong
to the {\it weak Morrey space} $w\cM^p_q=w\cM^p_q(\R^d)$ if the quasi-norm
\[
\|f\|_{w\cM^p_q}:=
\sup_{\gamma>0}
\|\gamma\chi_{\{|f|>\gamma\}}\|_{\cM^p_q}
\]
is finite.
\end{dfn}

Note that, by using the inequality $\gamma \chi_{\{|f|>\gamma\}} \le |f|$ for every
$\gamma>0$, we have $\cM^p_q\subseteq w\cM^p_q$. The inclusion properties of weak
Morrey spaces, generalized Morrey spaces, generalized weak Morrey spaces, and their
necessary conditions were discussed in \cite{GHLM}. In particular, for the case of
Morrey spaces and weak Morrey spaces, the results can be stated as follows:

\bigskip

\begin{thm}\label{thm:GHLM}{\rm \cite{GHLM}}
For $1\le p_1\le p_2\le q<\infty$, the following inclusion holds:
\begin{align*}
	w\cM^{p_2}_q \subseteq w\cM^{p_1}_q.
\end{align*}
Further, if $p_1<p_2$, then
\begin{align*}
w\cM^{p_2}_q \subseteq \cM^{p_1}_q.
\end{align*}
\end{thm}

In addition to the above inclusion relations of Morrey spaces, we have the
following theorems.

\bigskip

\begin{thm}\label{thm-300716}
Let $1\le p_1<p_2<q<\infty$. Then each of the following inclusions is proper:
\begin{enumerate}
\item[{\rm (i)}] $\cM^{p_2}_q \subseteq  \cM^{p_1}_q$;
\item[{\rm (ii)}] $w\cM^{p_2}_q \subseteq \cM^{p_1}_q$;
\item[{\rm (iii)}] $w\cM^{p_2}_q \subseteq w\cM^{p_1}_q$.
\end{enumerate}
\end{thm}

\bigskip

\begin{thm}\label{thm-300516-1}
Let $1\le p\le q<\infty$. Then the inclusion $\cM^p_q \subseteq w\cM^p_q$ is proper.
\end{thm}	

\bigskip

\begin{rem}
The claim about the proper inclusion $\cM^{p_2}_q \subseteq  \cM^{p_1}_q$
is stated in \cite[p. 2]{GHLM} without proof. We shall see the detailed
explanation of this claim in the proof of Theorem \ref{thm-300716}(i).
In \cite[Remark 2.4]{GHLM}, the authors refer to \cite{GHSS} for the proper
inclusion  between the generalized Morrey space $L^{1, \phi}$ and the
corresponding weak type space $wL^{1.\phi}$, where $\phi(t)=t^{-1}\log(3+t)$.
Since $L^{1,\phi} \neq \mathcal{M}^1_q$ for this choice of $\phi$, Theorem
\ref{thm-300516-1} can be seen as a complement of the result in \cite{GHSS}.
\end{rem}

We also obtain the following necessary conditions for inclusion of Morrey spaces
and weak Morrey spaces which can be seen as a refinement of some necessary
conditions given in \cite{GHLM}.

\bigskip

\begin{thm}\label{thm-3216-1}
Let $1\le p_i\le q_i<\infty$ for $i=1, 2$. Then
the following implications hold:
\begin{enumerate}
\item[{\rm (i)}]  $\cM^{p_2}_{q_2} \subseteq \cM^{p_1}_{q_1}$ implies
$q_1=q_2$ and $p_1\le p_2$;
\item[{\rm (ii)}] $w\cM^{p_2}_{q_2} \subseteq w\cM^{p_1}_{q_1}$ implies
$q_1=q_2$ and $p_1\le p_2$;
\item[{\rm (iii)}] $w\cM^{p_2}_{q_2} \subseteq \cM^{p_1}_{q_1}$  implies
$q_1=q_2$ and $p_1<p_2$.
\end{enumerate}
\end{thm}

\bigskip

\begin{rem}
A necessary and sufficient condition for inclusion of Morrey spaces on a bounded
domain can be found in \cite[Theorem 2.1]{P} and \cite{P2}.
The case of Morrey spaces on $\R^d$
is mentioned in \cite[Eq. (3.9)]{HS} and the authors refer to \cite[Satz 1.6]{R}.
However, we do not have the access to the paper, so that we do not know how the
proof goes. See also \cite[Corollary 3.14]{HS} for weighted version
of Theorem \ref{thm-3216-1}. Here we present a proof of the necessary and sufficient condition
for the inclusion property, which is different from and simpler than that in \cite{P}.
\end{rem}

The organization of this paper is as follows. In the next section, we prove that
for $1\le p_1<p_2<q<\infty$ the set $\cM^{p_1}_q \setminus \cM^{p_2}_q$ is not empty.
By the same example, we also show that for $1\le p_1<p_2<q$ the inclusion
$w\cM^{p_2}_q \subseteq w\cM^{p_1}_q$ is proper. In Section 3, we give the proof
of Theorem \ref{thm-300516-1} using the unboundedness of the Hardy-Littlewood
maximal operator on Morrey spaces of exponent 1. The proof of Theorem
\ref{thm-3216-1} is given in the last section. Throughout this paper, we denote
by $C$ a positive constant which is independent of the function $f$ and its value
may be different from line to line.

\section{The proof of Theorem \ref{thm-300716}}

We shall first prove Theorem \ref{thm-300716} (i) by constructing a function
which belongs to $\cM^{p_1}_q$ but not to $\cM^{p_2}_q$, for $1\le p_1 <p_2<q<\infty$.

\begin{proof}[Proof of Theorem \ref{thm-300716} (i)]
Let $1\le p_1<p_2<q<\infty$ and $\beta:=\frac{d(p_1+p_2)}{2q}$. Then we have
\begin{align}\label{eq:1}
\frac{dp_1}{q}<\beta<\frac{dp_2}{q}
\end{align}
and
\begin{align*}
d-\beta=\frac{d(q-p_1)+d(q-p_2)}{2q}>0.
\end{align*}
Define $g(x):=\chi_{B(0,1)}(x)+\chi_{\mathbb{R}^n \setminus B(0,1)}(x)|x|^{-\beta}$.
Then, for each $k\in \mathbb{N}$, we choose $r_k\in (k,k+1)$ such that
\[
\int_{B(0,k+1)\setminus B(0,k)}
g(x) \ dx
= |B(0,r_k)\setminus B(0,k)|.
\]
Next define
\begin{align}\label{eq:300716-1}
f(x):=\chi_{B(0,1)}(x)
+
\sum_{k=1}^{\infty}
\chi_{B(0,r_k)\setminus B(0,k)}(x).
\end{align}
We shall show that $f\in \cM^{p_1}_q\setminus \cM^{p_2}_q$.
First observe that
\[
\int_{B(a,r)} |f(x)|^p dx \le
\int_{B(0,r)} |f(x)|^p\,dx
\]
for every $1\le p<\infty,\ a\in\mathbb{R}^d,\ r>0$.
Now, for $1\le p<\infty$ and $r>2$, we have
\begin{align*}
\int_{B(0,r)}
|f(x)|^{p} \ dx
=
\int_{B(0,r)}
|f(x)| \ dx
\le
\int_{B(0,2r)}
g(x) \ dx,
\end{align*}
so
\begin{align}\label{eq:201016-1}
\int_{B(0,r)}
|f(x)|^{p} \ dx
\le \int_{B(0,2r)}
|x|^{-\beta} \ dx =Cr^{d-\beta}
\end{align}
and
\begin{align}\label{eq:201016-2}
\int_{B(0,r)}
|f(x)|^{p} \ dx
\ge \int_{B(0,r)\setminus B(0,1)}
|x|^{-\beta} \ dx =C(r^{d-\beta}-1)
\ge C \left(1-\frac{1}{2^{d-\beta}}\right)
r^{d-\beta}.
\end{align}
Therefore, by substituting $p=p_1$ into \eqref{eq:201016-1} and recalling
\eqref{eq:1}, we have
\begin{align}\label{eq:161102-1}
|B(0,r)|^{\frac{1}{q}-\frac{1}{p_1}}
\left(
\int_{B(0,r)} |f(x)|^{p_1} \ dx
\right)^{\frac{1}{p_1}}
\le C r^{\frac{d}{q}-\frac{d}{p_1}}
r^{\frac{d}{p_1}-\frac{\beta}{p_1}}
=C r^{\frac{d}{q}-\frac{\beta}{p_1}}\le C.
\end{align}
On the other hand, for each $r\le 2$, we have
\begin{align}\label{eq:3}
|B(0,r)|^{\frac{1}{q}-\frac{1}{p_1}}
\left(
\int_{B(0,r)} |f(x)|^{p_1} \ dx
\right)^{\frac{1}{p_1}}
\le C
r^{\frac{d}{q}-\frac{d}{p_1}}
\left(
\int_{B(0,r)} |f(x)|^{p_1}  dx
\right)^{\frac{1}{p_1}}
\le C
r^{\frac{d}{q}}\le C.
\end{align}
By combining \eqref{eq:161102-1} and \eqref{eq:3}
we conclude that $f\in \cM^{p_1}_q$.

Meanwhile, by substituting $p=p_2$ into \eqref{eq:201016-2}, we have
\begin{align*}
|B(0,r)|^{\frac{1}{q}-\frac{1}{p_2}}
\left(
\int_{B(0,r)}
|f(x)|^{p_2} \ dx
\right)^{\frac{1}{p_2}}
\ge C
r^{\frac{d}{q}-\frac{d}{p_2}}
r^{\frac{d-\beta}{p_2}}
=C
r^{\frac{d}{q}-\frac{\beta}{p_2}}.
\end{align*}
Since $\frac{d}{q}-\frac{\beta}{p_2}>0$, we have
\begin{align*}
\sup_{a\in \mathbb{R}^d, r>0}
|B(a,r)|^{\frac{1}{q}-\frac{1}{p_2}}
\left(
\int_{B(a,r)}
|f(x)|^{p_2} \ dx
\right)^{\frac{1}{p_2}}
&\ge C
\sup_{r> 2}
|B(0,r)|^{\frac{1}{q}-\frac{1}{p_2}}
\left(
\int_{B(0,r)}
|f(x)|^{p_2} \ dx
\right)^{\frac{1}{p_2}}\\
&\ge C
\sup_{r> 2}
r^{\frac{d}{q}-\frac{\beta}{p_2}}=
\infty.
\end{align*}
Thus $f\notin \cM^{p_2}_q$, and we are done.
\end{proof}

Theorem \ref{thm-300716} (ii) and (iii) are proved by using the function $f$
from the proof of Theorem \ref{thm-300716} (i) and its relation with the
characteristic function of its level set. The detailed proof goes as follows:

\begin{proof}[Proof of Theorem \ref{thm-300716} (ii)-(iii)]
For $1\le p_1<p_2<q<\infty$, let $f$ be defined by \eqref{eq:300716-1}.
Observe that
\begin{align*}
\chi_{\{|f|>\gamma\}}
=
\begin{cases}
0, \quad &\gamma\ge 1,\\
f, \quad &\gamma \in (0,1).
\end{cases}
\end{align*}
This together with the fact that $f\notin \cM^{p_2}_q$ gives
\begin{align*}
\|f\|_{w\cM^{p_2}_q}
=
\sup_{\gamma \in (0,1)}
\gamma\|\chi_{\{|f|>\gamma\}}\|_{\cM^{p_2}_q}
=
\sup_{\gamma \in (0,1)}
\gamma\|f\|_{\cM^{p_2}_q}
=
\|f\|_{\cM^{p_2}_q}=\infty,
\end{align*}
and hence $f\in \cM^{p_1}_q\setminus w\cM^{p_2}_q$. Thus we have
shown that $w\cM^{p_2}_q\subseteq  \cM^{p_1}_q$ is a proper inclusion.
Since $\cM^{p_1}_q\subseteq w\cM^{p_1}_q$, we  also have
$f\in w\cM^{p_1}_q\setminus w\cM^{p_2}_q$, so the inclusion (iii) is proper.
\end{proof}

\section{The proof of Theorem \ref{thm-300516-1}}

In order to prove Theorem \ref{thm-300516-1}, we need the following lemma.

\bigskip

\begin{lem}\label{lem:20416-1}
Let $1\le p\le q<\infty$. Then
\[
\|f\|_{\cM^{p}_q}
=\||f|^p\|_{\cM^{1}_{\frac{q}{p}}}^{\frac{1}{p}}
\]
for every $f\in \cM^p_q$ and
\[
\|f\|_{w\cM^{p}_q}
=\||f|^p\|_{w\cM^{1}_{\frac{q}{p}}}^{\frac{1}{p}}
\]
for every $f\in w\cM^p_q$.
\end{lem}
\begin{proof}
We calculate
\[
\|f\|_{\cM^p_q}
=
\sup_{B}
\left(
|B|^{\frac{p}{q}-1}
\int_{B} |f(x)|^p \ dx
\right)^{\frac{1}{p}}
=
\||f|^p\|_{\cM^1_{\frac{q}{p}}}^{\frac{1}{p}}.
\]
By applying the first identity for $\chi_{\{|f|> \gamma^{\frac{1}{p}} \}}$, we have
\[
\||f|^p\|_{w\cM^1_{\frac{q}{p}}}^{\frac{1}{p}}
=
\sup_{\gamma>0}
\gamma^{\frac{1}{p}}
\|\chi_{\{|f|^p>\gamma \}}\|_{\cM^{1}_{\frac{q}{p}}}^{\frac{1}{p}}
=
\sup_{\gamma>0}
\gamma^{\frac{1}{p}}
\|\chi_{\{|f|>\gamma^{\frac{1}{p}} \}}\|_{\cM^{p}_q}
=\|f\|_{w\cM^p_q},
\]
as desired.
\end{proof}

We also use the following fact about the unboundedness of the
Hardy-Littlewood maximal operator $M$ on Morrey spaces of exponent 1.
The operator $M$ maps a locally integrable function $f$ to $Mf$ which
is given by
\[
Mf(x):=\sup_{r>0} \frac{1}{|B(x,r)|}\int_{B(x,r)} |f(y)|\,dy,\quad x\in
\R^d.
\]

\bigskip

\begin{lem}\label{lem:M}
The Hardy-Littlewood maximal operator $M$ is not bounded on the Morrey
space $\cM^1_q$ for $1<q<\infty$.
\end{lem}

\bigskip

\begin{rem}
Lemma \ref{lem:M} is a consequence of a necessary condition of the
boundedness of $M$ on generalized Orlicz-Morrey spaces given in
\cite[Corollary 5.3]{N}. The Morrey space $\cM^p_q$ in this paper
is recognized as the Orlicz-Morrey space $L^{(\Phi,\phi)}$ with
$\Phi(t)=t^{p}$ and $\phi(t)=t^{-\frac{1}{q}}$. Based on
\cite[Corollary 5.3]{N}, the maximal operator $M$ is bounded on
$L^{(\Phi,\phi)}$ if and only if $\Phi\in \nabla_2$ (that is, $\Phi(r)\le
\frac{1}{2k}\Phi(kr)$ for some $k\ge 1$). Clearly $\Phi(t)=t\notin
\nabla_2$.
\end{rem}

\medskip

Now, we are ready to prove Theorem \ref{thm-300516-1}.

\begin{proof}[Proof of Theorem \ref{thm-300516-1}]
Let $1\le p\le q$. If $p=q$, then $f(x):=|x|^{-\frac{d}{q}} \in w\cM^p_q \setminus
\cM^p_q$. So assume that $p<q$
and write $r=\frac{q}{p}$. In view of Lemma \ref{lem:20416-1}, it suffices for us
to prove that $\cM^1_r \subset w\cM^1_r$ properly.
Suppose to the contrary that $\cM^1_r=w\cM^1_r$. Since the Hardy-Littlewood
maximal operator $M$ is bounded from $\mathcal{M}^1_r$ to $w\mathcal{M}^1_r$,
we obtain
\[
\|Mg\|_{w\cM^1_r} \le C\,\|g\|_{\cM^1_r},
\]
for every $g\in \cM^1_r$. Meanwhile, by the Closed Graph Theorem, there must exist a
constant $C'>0$ such that
$$
\|Mg\|_{\cM^1_r}\le C'\|Mg\|_{w\cM^1_r}
$$
for every $g\in \cM^1_r$. Combining the two inequalities, we obtain
$$
\|Mg\|_{\cM^1_r}\le C\|g\|_{\cM^1_r}
$$
for every $g\in \cM^1_r$. This tells us that $M$ is bounded on $\cM^1_r$, which
contradicts Lemma \ref{lem:M}. Therefore, $w\cM^1_r \setminus \cM^1_r \neq \emptyset$,
as desired.
\end{proof}

To conclude this section, we write a proposition which gives us a subset of weak
Morrey spaces with norm equivalence between the Morrey norm $\|\cdot\|_{\cM^p_q}$
and the weak Morrey quasi-norm $\|\cdot\|_{w\cM^p_q}$.

\bigskip

\begin{prop}
Let $1\le p< q<\infty$. Suppose that $f$ is a positive radial decreasing function
in $w\mathcal{M}^p_q(\mathbb{R}^d)$.
Then $f\in \mathcal{M}^p_q(\mathbb{R}^d)$ with
	\[
	\|f\|_{w\mathcal{M}^p_q}
	\le
	\|f\|_{\mathcal{M}^p_q}
	\le
	\left(
	\frac{q \omega_{d-1}}{d(q-p)|B(0,1)|}
	\right)^{\frac{1}{p}}
	\|f\|_{w{\mathcal{M}}^p_q},
	\]
that is, $\|f\|_{w\cM^p_q} \sim \|f\|_{\cM^p_q}$.
\end{prop}

\begin{proof}
Recall that, since $\gamma\chi_{\{|f|>\gamma \}}\le |f|$ for
every $\gamma>0$, we have $\|f\|_{w\cM^p_q} \le \|f\|_{\cM^p_q}$.
Next, let $x\in \mathbb{R}^d$. Since
	$
	\{
	y\in B(0,|x|):
	f(y)>f(x)
	\}
	=B(0,|x|)
	$,
we have
	\begin{align*}
	f(x)&=
	\frac{f(x)|\{y \in B(0,|x|): f(y)>f(x) \}|^{\frac{1}{p}}}{|B(0,|x|)|^{\frac{1}{p}}}
	\\
	&\le
	\frac{|B(0,|x|)|^{\frac{1}{p}-\frac{1}{q}}\|f\|_{w\mathcal{M}^p_q}}{|B(0,|x|)|^{\frac{1}{p}}}	
	\\
	&=
	|B(0,1)|^{-\frac{1}{q}}
	\|f\|_{w\mathcal{M}^p_q}
	|x|^{-\frac{d}{q}}.
	\end{align*}
By combining the last estimate and
	\[
	\||x|^{-\frac{d}{q}}\|_{\mathcal{M}^p_q}=|B(0,1)|^{\frac{1}{q}}\left(
	\frac{q \omega_{d-1}}{d(q-p)|B(0,1)|}
	\right)^{\frac{1}{p}},
	\]
where $\omega_{d-1}$ is the surface area of the unit sphere $\mathbb{S}^{d-1}$, we get
	\begin{align*}
	\|f\|_{\mathcal{M}^p_q}
	\le
	(|B(0,1)|^{-\frac{1}{q}}
	\||x|^{-\frac{d}{q}}\|_{\mathcal{M}^p_q})
	\|f\|_{wM^p_q} =
	\left(
		\frac{q\omega_{d-1}}{d(q-p)|B(0,1)|}
	\right)^{\frac{1}{p}}
	\|f\|_{wM^p_q}.
	\end{align*}
Hence $\|f\|_{w\cM^p_q} \sim \|f\|_{\cM^p_q}$.
\end{proof}


\section{The proof of Theorem \ref{thm-3216-1}}

\begin{proof}[Proof of Theorem \ref{thm-3216-1} {\rm (i)}]
It follows from the inclusion $\cM^{p_2}_{q_2} \subseteq \cM^{p_1}_{q_1}$ that
\[
\|\chi_{B(0,r)}\|_{\cM^{p_1}_{q_1}}
\le C
\|\chi_{B(0,r)}\|_{\cM^{p_2}_{q_2}},
\]
for every $r>0$. Therefore
\[
r^{\frac{d}{q_1}-\frac{d}{q_2}}\le C
\]
for every $r>0$, which implies that $q_1=q_2$.
Now choose $\epsilon\in\bigl(0,\min\{\frac{dp_1}{q_1}, \frac{dp_2}{q_2}\}\bigr)$.
For $j\in\N$, define
$h_j(x):=\chi_{ \{ j\le |x|\le j+j^{-\epsilon} \}}(x)$, and for
$K\in\N$ write
$f(x):=\chi_{\{ 0\le |x|<1 \}}(x)+\sum_{j=1}^{K} h_j(x)$.
Then
\begin{align}\label{eq:9217-1}
\|f\|_{\cM^{p_1}_{q_1}}
&\ge
|B(0,K+K^{-\epsilon})|^{\frac{1}{q_1}-\frac{1}{p_1}}
\left(
\int_{B(0,K+K^{-\epsilon})}
|f(x)|^{p_1} \ dx
\right)^{\frac{1}{p_1}}
\nonumber
\\
&\ge
C (K+K^{-\epsilon})^{\frac{d}{q_1}-\frac{d}{p_1}}
(K+K^{-\epsilon})^{\frac{d}{p_1}-\frac{\epsilon}{p_1}}
=C (K+K^{-\epsilon})^{\frac{d}{q_1}-\frac{\epsilon}{p_1}}.
\end{align}
Meanwhile, for each $L\in \mathbb{N},\ L\le K$, we observe that
\[
|B(0,L+L^{-\epsilon})|^{\frac{1}{q_2}-\frac{1}{p_2}}\left(\int_{B(0,L+L^{-\epsilon})}
|f(x)|^{p_2}dx\right)^{\frac{1}{p_2}}
\le C(L+L^{-\epsilon})^{\frac{d}{q_2}-\frac{\epsilon}{p_2}}.
\]
Hence,
\begin{equation}\label{eq:9217-2}
\|f\|_{\cM^{p_2}_{q_2}}
\le C
(K+K^{-\epsilon})^{\frac{d}{q_2}-\frac{\epsilon}{p_2}}.
\end{equation}
By combining \eqref{eq:9217-1}, \eqref{eq:9217-2}, $q_1=q_2$,
and $\|f\|_{\cM^{p_1}_{q_1}}\le C\|f\|_{\cM^{p_2}_{q_2}}$, we get
\[
(K+K^{-\epsilon})^{\frac{\epsilon}{p_2}-\frac{\epsilon}{p_1}}
\le C.
\]
As this holds for every $K\in\mathbb{N}$, we conclude that $p_1\le p_2$.
\end{proof}
\begin{rem}
Note that the difference between the proof of Theorem \ref{thm-3216-1} {\rm (i)} and
\cite[Remark 3.4]{GHLM} is that we do not assume $p_1\le p_2$.
\end{rem}
\begin{proof}[Proof of Theorem \ref{thm-3216-1} {\rm (ii)}]
By arguing as in the proof of Theorem \ref{thm-3216-1} {\rm (i)} and using
the identities $\|\chi_{B(0,r)}\|_{w\cM^{p_1}_{q_1}}=|B(0,r)|^{\frac{1}{q_1}}$ and
$\|\chi_{B(0,r)}\|_{w\cM^{p_2}_{q_2}}=|B(0,r)|^{\frac{1}{q_2}}$,
we have  $q_1=q_2$.
Assume to the contrary that $p_1>p_2$. Define $f$ by \eqref{eq:300716-1}.
By a similar argument as in the proof of Theorem 1.3 (ii)-(iii),
we have $f\in w\cM^{p_2}_{q_2}$ but $f\notin w\cM^{p_1}_{q_1}$, which
contradicts $w\cM^{p_2}_{q_2}\subseteq w\cM^{p_1}_{q_1}$. Hence $p_1\le p_2$.
\end{proof}
\begin{rem}
Observe that unlike \cite[Theorem 4.4 and Remark 4.5]{GHLM}, the condition
$p_1\le p_2$ is not assumed in Theorem \ref{thm-3216-1} {\rm (ii)}.
\end{rem}
\begin{proof}[Proof of Theorem \ref{thm-3216-1} {\rm (iii)}]
Since $\cM^{p_2}_{q_2} \subseteq w\cM^{p_2}_{q_2}$, we have
$\cM^{p_2}_{q_2} \subseteq \cM^{p_1}_{q_1}$.
Therefore, by virtue of Theorem \ref{thm-3216-1} {\rm (ii)}, we have $q_1=q_2$
and $p_1\le p_2$. Now, assume to the contrary that $p_1= p_2$.
According to Theorem 1.4,
there exists $f_0 \in w\cM^{p_2}_{q_2}$ such that $f_0\notin \cM^{p_1}_{q_1}$.
This contradicts $w\cM^{p_2}_{q_2} \subseteq \cM^{p_1}_{q_1}$.
Thus $p_1<p_2$, as desired.
\end{proof}

\noindent{\bf Acknowledgement}. The first and second authors are supported by
P3MI -- ITB Program 2017. We would like to thank the referees for their useful
comments on the earlier version of this paper.

\end{document}